\documentclass[12pt]{amsart}
\usepackage{graphicx} 

\usepackage{vmargin,amsmath,amsfonts,amssymb,amsthm,xcolor,amscd,latexsym,xspace,enumerate}
\usepackage[colorlinks=true,linkcolor=blue,citecolor=red]{hyperref}
\usepackage{dirtytalk}
\usepackage{tikz}
\usepackage[backend=biber]{biblatex}
\usepackage{verbatim}
\renewbibmacro{in:}{} 
\newtheorem{theorem}{Theorem}
\newtheorem{corollary}[theorem]{Corollary}
\newtheorem{lemma}[theorem]{Lemma}
\newtheorem{proposition}[theorem]{Proposition}

\newcommand{\supp}{\operatorname{supp}}

\DeclareMathOperator*{\esssup}{ess\,sup}

\begin{document}
\title[Delsarte extremal problem and convolution roots]{Delsarte-Type Extremal Problems and Convolution Roots on Homogeneous Spaces}
\author{Mita D. Ramabulana}  
\address{Mita D. Ramabulana \endgraf
Department of Mathematics and Applied Mathematics \endgraf
University of Cape Town \endgraf
Private Bag X3, 7701, Rondebosch \endgraf
Cape Town, South Africa}
\email{mita.ramabulana@uct.ac.za}

\begin{abstract}
    For a locally compact group $G$ and compact subgroup $K$, we consider a Delsarte-type extremal problem for $G$-invariant positive definite kernels on the homogeneous space $G/K$, generalising a certain Tur\'an problem for isotropic positive definite kernels on the unit sphere $\mathbb{S}^d$ in $\mathbb{R}^{d+1}$. We exploit a correspondence between $G$-invariant kernels on $G/K$ and $K$-bi-invariant functions on $G$ to show that the Delsarte-type problem on a homogeneous space is equivalent to a Delsarte-type problem for $K$-bi-invariant functions on its group $G$ of transformations. We use this correspondence to show the existence of an extremal function for the Delsarte problem on the homogeneous space. In the case where $(G,K)$ is a compact Gelfand pair, we show the existence of $K$-bi-invariant convolution roots for positive definite $K$-bi-invariant functions, consequently obtaining the existence of a $G$-invariant convolution root for $G$-invariant positive definite kernels.
\end{abstract}

\maketitle
\maketitle

\section{Introduction}
Delsarte-type extremal problems for positive definite functions in the general setting of locally compact Abelian groups have been studied in \cite{kolrev, Rvsz2009TurnsEP, marcell-zsuzsa,berdyshevarevesz, ramabulana, berdrevram}. In the non-commutative \textit{compact} group setting, the Delsarte scheme is used in \cite{kolmatwei} to study the problem of mutually unbiased bases. One of the aims of this paper is to motivate the study of Delsarte-type problems on non-commutative groups. In particular, starting with a Delsarte-type problem on a homogeneous space $G/K$, we pose an equivalent Delsarte-type problem on its group $G$ of transformations. The group of transformations is, in general, non-commutative, so that we deviate from the commutative group setting quite naturally. Nevertheless, commutativity still plays a role since the collection of admissible functions of our extremal problem are integrable $K$-bi-invariant functions, the collection of which forms a commutative algebra under convolution.

Our point of departure is the Tur\'an problem for continuous positive definite kernels on the unit sphere in Euclidean space, which we now describe. We follow \cite{ gneiting, gneitingsupplement}. For an integer $d \ge 2$, let $$\mathbb{S}^d := \{x \in \mathbb{R}^{d+1}: \|x\|=1 \}$$ denote the unit sphere in $\mathbb{R}^{d+1}$. A function $f: \mathbb{S}^d \times \mathbb{S}^d \to \mathbb{C}$ is \textit{positive definite} if
\begin{equation*}\label{spherepd}
  \sum_{i=1}^{n}\sum_{j=1}^{n}f(x_i,x_j)c_i\overline{c_j} \ge 0
\end{equation*}
for all $n \in \mathbb{N}$, for all points $x_{i} \in \mathbb{S}^d$, and for all complex numbers $c_i \in \mathbb{C}$, with $1\le i \le n$. Let $\theta(x,y) = \arccos (\langle x, y \rangle)$ denote the great-circle distance between points $x$ and $y$ in $\mathbb{S}^d$, where $\langle x, y \rangle$ denotes the inner product of $x$ and $y$ in $\mathbb{R}^d$. A function $f: \mathbb{S}^d \times \mathbb{S}^d \to \mathbb{C}$ is called \textit{isotropic} if there exists a function $\psi: [0, \pi] \to \mathbb{C}$ such that
\begin{equation*}
    f(x,y) = \psi (\theta(x,y)).
\end{equation*}
The fact that an isotropic function $f$ factors as $f=\psi \circ \theta$ simply means that its values depend only on the great-circle distance between its inputs. From the definition, it is straightforward to produce examples of isotropic functions.\par
To be consistent with the notation in \cite{gneiting, gneitingsupplement}, let $\Psi^d$ denote the collection of continuous functions $\psi: [0, \pi] \to \mathbb{R}$ with $\psi(0)=1$ and such that the associated isotropic function $\psi \circ \theta:\mathbb{S}^d \times \mathbb{S}^d \to \mathbb{R}$ is positive definite. For $0 < c \le \pi$, let $\Psi^d_c$ denote the collection of functions $\psi \in \Psi^d$ such that $\psi(t)=0$ for $t \ge c$. 
We now state the Tur\'an problem for a spherical cap of radius $0< c \le \pi$. The problem is that of computing
\begin{equation}\label{spherturan}
    \mathcal{T}_{\mathbb{S}^d}(c):=\sup_{\psi \in \Psi^d_c}\int_{\mathbb{S}^d}\psi(\theta(x,y))\textup{d}y,
\end{equation}
where $x$ is an arbitrary point on $\mathbb{S}^d$ and the integral is with respect to the surface measure on the sphere. 

The Tur\'an problem for isotropic positive definite kernels on spheres is an open problem posed in \cite{gneitingsupplement} in connection with applications in spatial statistics. In this setup, the positive definite kernels are correlation functions, the sphere models the surface of the planet Earth, and the data comes from geophysical, meteorological and climatological applications. 

In particular, if the kernel \( \Phi(x, y) = \text{Cov}(Z(x), Z(y)) \) is the covariance of a real-valued random field \( \{Z(x): x \in \mathbb{S}^d \} \) on \( \mathbb{S}^d \), where, for each $x \in \mathbb{S}^d$, $Z(x)$ is a random variable, then for a fixed $x \in \mathbb{S}^d$, the quantity
\[
\int_{\mathbb{S}^d} \Phi(x, y) \, \mbox{d}y,
\]
represents the total covariance or interaction of \( Z(x) \) with the rest of the field.
So, the expression

\[
\sup_\Phi \int_{\mathbb{S}^d} \Phi(x, y) \, \mbox{d}y
\]
where $\Phi$ runs through the collection of such kernels, is the maximum total covariance that can be concentrated locally around \( x \), without violating the constraint of positive definiteness, and while ensuring zero correlation beyond radius \( c \).

The above Tur\'an problem for a spherical cap on the unit sphere motivates us to study a related Tur\'an problem on the non-commutative group $\mathrm{SO}(d+1)$ of $d \times d$ matrices having determinant $1$, after the observation that the sphere is a homogeneous space in the sense that $\mathbb{S}^d$ can be identified with $\mathrm{SO}(d+1)/\mathrm{SO}(d)$. 

Apart from applications in statistics, a variant of our Delsarte-type problem on non-commutative groups appears independently in \cite{maxmilian} in connection to sphere packing in metric geometries beyond Euclidean space. 

Another aim of the paper is to study the existence of convolution roots for positive definite kernels on homogeneous spaces. In particular, we extend a result of \cite{ziegel} on the existence of isotropic convolutions roots for positive definite functions on the sphere to homogeneous spaces $G/K$ where the pair $(G, K)$ forms a Gelfand pair. Our approach is again to show that the problem of the existence of convolution roots for kernels on the homogeneous space is equivalent to the problem of the existence of convolution roots on its group of transformations.

\section{Preliminaries and Notation}
In this section, we recall some notions and notation. Let $G$ be a locally compact group with Haar measure $\lambda_{G}$. For each $g \in G$, the assignment $\mu_g(U) := \lambda_{G}(Ug)$ for all Borel measurable subsets $U \subset G$ defines a left Haar measure, which is related to $\lambda_{G}$ by a constant $\Delta_{G} (g) > 0$ such that $\mu_g = \Delta_{G}(g)\lambda_{G}$. Let $\mathbb{R}^{\times}$ denote the multiplicative group of positive real numbers. The function $\Delta_{G}:G \to \mathbb{R}^{\times}$ is a continuous group homomorphism called the modular function of $G$ \cite[Proposition 2.24]{folland}). For any $g \in G$, let $R_{g}f$ denote the right translate of $f$ by $g$, given by $R_{g}f(x) = f(xg)$ for all $x \in G$. According to \cite[Proposition 2.24]{folland},  
\begin{equation*}
    \int_{G}R_{g}f(x)\textup{d}\lambda_{G}(x) = \int_{G}f(xg)\textup{d}\lambda_{G}(x) = \Delta_{G}(g^{-1})\int_{G}f(x)\textup{d}\lambda_{G}(x).
\end{equation*}
In particular, if $K$ is a compact subgroup of $G$ and $k \in K$, then $\Delta_{G}(k^{-1}) = 1$, and we have
\begin{equation}\label{compactmodularfuction}
    \int_{G}f(xk)\textup{d}\lambda_{G}(x) = \Delta_{G}(k^{-1})\int_{G}f(x)\textup{d}\lambda_{G}(x) = \int_{G}f(x)\textup{d}\lambda_{G}(x).
\end{equation}
If $K$ is a subgroup of $G$, then $K$ has a left-invariant Haar measure $\lambda_{K}$, unique to a positive constant. If $K$ is compact, then $\lambda_{K}(K) < \infty$, and so we shall always normalise the Haar measure on a compact (sub)group $K$ so that $\lambda_{K}(K) = 1$.

For $1 \le p < \infty$, we denote by $L^{p}(G)$ the usual Banach space of complex-valued functions on $G$ satisfying
\begin{equation*}\label{pnorm}
    \|f\|_{L^{p}(G)}:= \left(\int_{G}|f(g)|^p\mbox{d}\lambda_{G}(g) \right)^{1/p} < \infty.
\end{equation*}
We denote by $L^{\infty}(G)$ the collection of essentially bounded complex-valued functions, i.e., functions $f: G \to \mathbb{C}$ satisfying 
\begin{equation*}
    {\|f\|_{L^{\infty}(G)} := \esssup_{g \in G}|f(g)| < \infty}.
\end{equation*}
The support of a continuous function $f: G \to \mathbb{C}$ is the set
\begin{equation*}
    \supp f := \overline{\{g \in G: f(g) \ne 0\}}.
\end{equation*}
The positive part of $f: G \to \mathbb{R}$ is given by 
\begin{equation*}
    f_{+}(g) = \max\{{f(g), 0\}} \textup{ for all } g \in G,
\end{equation*}
and its negative part is given by 
\begin{equation*}
     f_{-}(g) = \max\{{-f(g), 0\}} \textup{ for all } g \in G.
\end{equation*}
The collection of continuous complex-valued functions on $G$ is denoted by $C(G)$, and its subset consisting of continuous compactly supported functions is denoted by $C_{c}(G)$. For any Borel measurable functions $f,g: G \to \mathbb{C}$, the convolution of $f$ and $g$ is given by 
\begin{equation*}
    f \ast g(x) := \int_{G}f(y)g(y^{-1}x)\textup{d}\lambda_{G}(y),
\end{equation*}
whenever the integral exists. 

Let $K$ be a compact subgroup of \(G\) with Haar measure $\lambda_{K}$. Then $G/K$ is a homogeneous space in the sense of \cite[Section 2.6, p.60]{folland}, and has a $G$-invariant Radon measure \cite[Corollary 2.53]{folland}. We denote this measure by $\mu$, and note according to \cite[Theorem 2.51]{folland} that it satisfies

\begin{equation}\label{weylintegrationformula}
    \int_{G}f(g)\textup{d}\lambda_{G}(g) = \int_{G/K}\int_{K}f(gk)\textup{d}\lambda_{K}(k)\textup{d}\mu(gK)
\end{equation}    
for all $f \in C_c(G)$. 

A function $f: G \to \mathbb{C}$ is positive definite if the inequality
		\begin{equation*}
			\sum_{i = 1}^{n}\sum_{j=1}^{n}c_{i}\overline{c_{j}}f(g_{i}^{-1}g_{j}) \ge 0 
		\end{equation*}
		holds for all choices of $n \in \mathbb{N}$, $c_{i} \in \mathbb{C}$, and $g_{i} \in G$ with $1 \le i \le n$.
Elementary properties of positive definite functions can be found in many texts on harmonic analysis on groups, for instance \cite{sasvari, folland, dijk}. We record a few here. 

Let $f: G \to \mathbb{C}$ be a positive definite function. Then $f(e) \ge 0$ and $|f(g)| \le f(e)$ for all $g \in G$. In particular, if $f$ is a positive definite function satisfying $f(e) = 0$, then $f$ is identically zero. If $g$ is another positive definite function, then $af + bg$ is a positive definite function for all $a,b \ge 0$, and so is their pointwise product $fg$. A positive definite function is symmetric in the sense that \(f(g) = \overline{f(g^{-1})}\) for all \(g \in G\). In particular, if \(f\) is real-valued, the \(f(g) = f(g^{-1})\) for all \(g \in G\).

The reader may refer to \cite[Lemma 5.1.8]{dijk}, \cite[Theorem 1.3.2, Lemma 1.3.6, Theorem 1.4.1]{sasvari}, and \cite[p. 91]{folland} for proofs of these statements and other properties of positive definite functions. 

Related to the notion of a positive definite function is that of an integrally positive definite function. Firstly, following \cite[p. 79]{folland}, we define for a function $f: G \to \mathbb{C}$, the involution of $f$ by $$f^{\ast}(g) = \Delta_{G}(g^{-1})\overline{f(g^{-1})}$$ for all $g \in G$.
Following {\cite[Section 3.3, p. 83]{folland}} a function $f \in L^{\infty}(G)$ on a locally compact group $G$ is integrally positive definite if it satisfies
\begin{equation}\label{integrallypdexpression}
    \int_{G}(\varphi ^{\ast} \ast \varphi)f\textup{d}\lambda_{G} \ge 0
\end{equation}
for all $\varphi \in C_{c}(G)$.
In other words, a function $\varphi \in L^{\infty}(G)$ is integrally positive definite if 

\begin{equation*}
\int_{G}\int_{G}f(y^{-1}x)\varphi(x)\overline{\varphi(y)}\textup{d}\lambda_{G}(x)\textup{d}\lambda_{G}(y) =  \int_{G}(\varphi ^{\ast} \ast \varphi)f\textup{d}\lambda_{G}  \ge 0
\end{equation*}
for all $f \in C_{c}(G)$.

The following important proposition gives a standard procedure for constructing continuous positive definite functions. 
\begin{proposition}[{\cite[Corollary 3.16]{folland}}]\label{autocorrelation}
Let \(G\) be a locally compact group. If $f \in L^2(G)$, let $ f^{\#}(x):=\overline{f(-x)}$; then $f \ast f^{\#}$ is a continuous positive definite function.
\end{proposition}

A function $\varphi: G \to \mathbb{C}$ is called $K$-bi-invariant if $\varphi(kgk')=\varphi(g)$ for all $g \in G$ and all $k, k' \in K$. We also say that a subset $U \subset G$ is $K$-bi-invariant if its characteristic function $\mathbf{1}_{U}$ is $K$-bi-invariant. This is equivalent to $U = KUK$. We say that a subset $U \subset G$ of $G$ is symmetric if $U^{-1}=U$. 

Let $C^{K}(G)$ denote the collection of continuous $K$-bi-invariant functions $f: G \to \mathbb{C}$ and \(C_{c}^{K}(G) = C_{c}(G)\cap C^{K}(G)\) the collection of continuous \(K\)-bi-invariant functions with compact support.  For any continuous compactly supported function $f \in C_{c}(G)$, the function
\begin{equation*}\label{Kaverage}
    f^{K}(g) := \int_{K}\int_{K}f(kgk')\textup{d}\lambda_{K}(k)\textup{d}\lambda_{K}(k')
\end{equation*}
for all $g \in G$, is a continuous compactly supported $K$-bi-invariant function. 
Define the projection operator $P^{K}: C_{c}(G) \ni f \mapsto f^{K} \in C_{c}^{K}(G)$. We denote by \(f^{K}_{+}\) the positive part of \(f^K\) and by \(f^{K}_{-}\) its negative part. 

We denote by $P(G,K)$ the collection of real-valued continuous positive definite $K$-bi-invariant functions on $G$, and by $P_{1}(G,K)$ the subset of $P(G,K)$ consisting of those functions $f: G \to \mathbb{R}$ satisfying $f(e)=1$.

Convolution makes \(C^{K}(G)\) into an algebra and \(C^{K}_{c}(C)\) is a subalgebra of \(C_{c}(G)\) under convolution. 
The pair \((G,K)\) said to be a Gelfand pair if the algebra \(C^{K}_{c}(G)\) is commutative \cite[Definition 6.1.1]{dijk}, and a compact Gelfand pair if \(G\) is compact. In particular, for a Gelfand pair \((G,K)\) the convolution of two \(K\)-bi-invariant functions is \(K\)-bi-invariant. A review of the harmonic analysis on Gelfand pairs can be found in \cite{dijk}, and a more detailed treatment can be found in \cite{dieudonne, Wolf2007}. We shall need the following definitions and facts.

Let \((G,K)\) be a Gelfand pair. A function \(\varphi: G \to \mathbb{C}\) is called a \textit{spherical function} if the functional \(\chi\) defined by 

\begin{equation*}
    \chi(f) = \int_{G}f(g)\varphi(g^{-1})\mbox{d}\lambda_{G}(g) \mbox{ for all } f \in C^{K}_{c}(G),
\end{equation*}
satisfies 
\begin{equation}\label{character}
    \chi(f \ast g) = \chi(f) \cdot\chi(g)
\end{equation}
for all \(f, g \in C_{c}^{K}(G)\). That is, \(\chi\) is a non-trivial character of the convolution algebra \(C^{K}_{c}(G)\). Let \(L^{1}(G)^{K}\) denote the convolution algebra of integrable \(K\)-bi-invariant functions on \(G\). Note that \(L^{1}(G)^{K}\) is commutative for a Gelfand pair \((G,K)\). A non-trivial character of \(L^1(G)^{K}\) is a linear functional \(\chi: L^{1}(G)^{K} \to \mathbb{C}\) satisfying equation \eqref{character}. The following theorem characterises the characters of \(L^1(G)^{K}\).

\begin{theorem}[{\cite[Theorem 6.1.7]{dijk}}]
    Let \(\varphi: G \to \mathbb{C}\) be a bounded spherical function. The mapping
    \begin{equation*}
        f \mapsto \chi(f)= \int_{G}f(g)\varphi(g^{-1})\textup{d}\lambda_{G}(g)
    \end{equation*}
    is a character of \(L^{1}(G)^{K}\), and each non-trivial character of \(L^1(G)^{K}\) is of this form.
\end{theorem}
Denote by \(Z\) the collection of positive definite spherical functions provided with the weak topology \(\sigma(L^{\infty}(G), L^{1}(G))\). The space \(Z\) is locally compact by \cite[Proposition 6.4.2]{dijk}, and we refer to it as the dual space or \textit{spherical dual} of the Gelfand pair \((G,K)\). The Fourier transform of a function \(f \in L^1(G)^{K}\) is defined to be
\begin{equation*}
    \widehat{f}(\varphi) = \int_{G}f(g)\varphi(g^{-1})\mbox{d}\lambda_{G}(g) \mbox{ for all } \varphi \in Z.
\end{equation*}
In particular, \(\widehat{f}\) is a function on \(Z\). For more on the harmonic analysis on Gelfand pairs, see \cite[Section 6.4]{dijk}.

Now, without assuming that \((G,K)\) is a Gelfand pair, we consider the space $G/K \times G/K$ with the product topology. We refer to functions $\Phi: G/K \times G/K \to \mathbb{C}$ as kernels on \(G/K\). Denote by $C(G/K \times G/K)$ the space of continuous complex-valued kernels on $G/K$, and $C^{G}(G/K \times G/K)$ its subset consisting of $G$-invariant kernels. For a kernel \(\Phi : G/K \times G/K \to \mathbb{R}\), define its positive part to be the kernel \(\Phi_{+}: G/K \times G/K \to \mathbb{R}\) given by

\begin{equation*}
    \Phi_{+}(xK,yK) = \max\{\Phi(xK,yK), 0\} \mbox{ for all } x,y \in G.
\end{equation*}
Similarly, the negative part of \(\Phi\) is defined by
\begin{equation*}
    \Phi_{-}(xK,yK):=\max\{-\Phi(xK,yK), 0\} \mbox{ for all } x,y \in G.
\end{equation*}
A kernel $\Phi: G/K \times G/K \to \mathbb{C}$ is positive definite if
    \begin{equation*}
        \sum_{i=1}^{n}\sum_{j=1}^{n}\Phi(x_iK,x_jK)a_i\overline{a_j} \ge 0
    \end{equation*}
for all $n \in \mathbb{N}$, for all complex numbers $a_1,..., a_n$, and for all points $x_1,...,x_n \in G$. Denote by $P_{1}(G/K)$ the collection of continuous positive definite kernels $$\Phi:G/K \times G/K\to\mathbb{R}$$ satisfying $\Phi(gK, gK) = 1$ for all $g \in G$. A kernel $\Phi: G/K \times G/K \to \mathbb{R}$ is $G$-invariant if 

\begin{equation*}
 \Phi(gxK,gyK) = \Phi(xK, yK) \textup{ for all } {x, y, g \in G}.
\end{equation*}
Denote by $P^{G}(G/K)$ the collection of continuous positive definite $G$-invariant kernels $\Phi: G/K \times G/K \to \mathbb{C}$, and let $P^{G}_{1}(G/K) =P _{1}(G/K) \cap P^{G}(G/K)$ be the collection of continuous positive definite \(G\)-invariant kernels with \(\Phi (xK,xK) = 1\) for all \(x \in G\). If a kernel \(\Phi: G/K \times G/K \to \mathbb{C}\) is \(G\)-invariant, then \(\Phi(xK,yK)=\Phi(eK, x^{-1}yK)\) for all \(x,y \in G\). We define the support \(\supp_{g} \Phi\) with respect to \(g \in G\) of a not necessarily \(G\)-invariant kernel \(\Phi\) to be
\begin{equation*}
    \supp_{g} \Phi:=\overline{\{x \in G: \Phi(gK, xK) \ne 0\}}.
\end{equation*}
If \(\Phi\) is \(G\)-invariant, then \(\supp_{g} \Phi = \supp_{h} \Phi\) for all \(g, h \in G\). In this case, we write \(\supp \Phi := \supp_{g} \Phi\) for any \(g \in G\), and call \(\supp \Phi\) the support of \(\Phi\). We note that for a \(G\)-invariant kernel \(\Phi\), we have that \(\Phi(xK,yK) = \Phi(eK, x^{-1}yK) \ne 0\) implies that \(x^{-1}y \in \supp \Phi\). 

If \(\Phi\) is a \(G\)-invariant kernel, then so are its positive and negative parts, \(\Phi_{+}\) and \(\Phi_{-}\), respectively. In particular, \(\supp \Phi_{\pm}\) is defined.

For kernels $\Phi_1, \Phi_2: G/K \times G/K \to \mathbb{R}$, their convolution is
\begin{equation}\label{spheicalconvolution2}
    \Phi_1\ast \Phi_2(xK,yK):= \int_{G/K}\Phi_1(xK,zK)\Phi_2(zK,yK)\textup{d}\mu (zK)
\end{equation}
for all $x, y \in G$.

\section{A Bijective Correspondence}
In this section, $G$ is a locally compact group and $K$ is a compact subgroup of $G$.
Let $\varphi: G \to \mathbb{C}$ be a continuous positive definite $K$-bi-invariant function. Define the kernel $\Phi: G/K \times G/K \to \mathbb{C}$ by
\begin{equation}\label{relation}
\Phi(gK,hK) := \varphi(g^{-1}h)
\end{equation} for all $g, h \in G$. In order that \eqref{relation} be well-defined, we need it to be independent of the particular choice of coset representatives $g, h \in G$. In other words, we need $\Phi(gkK,hk'K) = \Phi(gK,hK)$ for all $k, k' \in K$. However, using the $K$-bi-invaraince of $G$, we have
\begin{equation*}
    \Phi(gkK,hk'K)=\varphi((gk)^{-1}hk') = \varphi(k^{-1}g^{-1}hk') = \varphi(g^{-1}h) = \Phi(gK,hK).
\end{equation*}
Therefore, \eqref{relation} is well-defined. Note that $\Phi$ is continuous. Observe that for all $g \in G$ and all $xK, yK \in G/K$, we have
\begin{equation*}
\begin{split}
     \Phi(g(xK),g(yK)) =  \Phi(gxK, gyK) = \varphi((gx)^{-1}gy) = \varphi(x^{-1}g^{-1}gy) &= \varphi(x^{-1}y) \\ &=  \Phi(xK, yK).
\end{split}
\end{equation*}
Therefore, $\Phi$ is $G$-invariant. The calculation
\begin{equation*}
    \sum_{i = 1}^{n}\sum_{j = 1}^{n}\Phi(x_iK,x_jK)a_i\overline{a_j} =  \sum_{i = 1}^{n}\sum_{j =1}^{n}\varphi(x_i^{-1}x_{j})a_{i}\overline{a_{j}} \ge 0
\end{equation*}
for all $n \in \mathbb{N}$, for all $x_{1}, \ldots,x_n$, and for all $a_1, \ldots, a_n \in \mathbb{C}$, shows that $\Phi$ is positive definite.\par
Now, let us suppose that we are given a $G$-invariant continuous positive definite kernel $\Phi: G/K \times G/K \to \mathbb{C}$. Define 
\begin{equation}\label{backrelation}
    \varphi(g) := \Phi(K,gK)
\end{equation}
for all $g \in G$. The function $\varphi$ in \eqref{backrelation} is continuous. Using the $G$-invariance of $\Phi$, we have  
\begin{equation*}
    \varphi(kgk')=\Phi(K,kgk'K) = \Phi(k^{-1}K,gk'K) = \Phi(K,gK) = \varphi(g)
\end{equation*}
for all $k, k' \in K$. Therefore, $\varphi$ is a $K$-bi-invariant function. The calculation
\begin{equation*}
   \sum_{i = 1}^{n}\sum_{j=1}^{n}\varphi(x_i^{-1}x_j)a_{i}\overline{a_{j}} = \sum_{i = 1}^{n}\sum_{j=1}^{n}\Phi(K,x_i^{-1}x_jK)a_{i}\overline{a_{j}} = \sum_{i = 1}^{n}\sum_{j=1}^{n}\Phi(x_iK,x_{j}K)a_{i}\overline{a_{j}} \ge 0,
\end{equation*}
shows that $\varphi$ is positive definite. The discussion above shows the following theorem.
\begin{theorem}\label{correspondence}
    Let $G$ be a locally compact group and $K$ a compact subgroup of $G$. There is a bijective correspondence between the collection of continuous positive definite $G$-invariant kernels on $G/K \times G/K$ and the continuous positive definite $K$-bi-invariant functions on $G$. The correspondence is given by the relations
    \begin{equation*}
        \Phi(xK,yK) = \varphi(x^{-1}y),
    \end{equation*}
    \begin{equation*}
        \varphi(x) = \Phi(K,xK)
    \end{equation*}
for all $x, y \in G$.
\end{theorem}
Theorem \ref{correspondence} occurred to us quite naturally when we tried to connect the spherical Tur\'an problem on groups and on homogeneous spaces, but the statement can be found in \cite[p. 113]{dijk}, where it is mentioned without proof. We included the proof here for ease of the reader. 

Our aim in this section is to give some properties of this bijective correspondence. To that end, let $J:P(G,K) \to P^{G}(G/K)$ be the map defined by 
\begin{equation}
    J(\varphi)(xK,yK) = \varphi(x^{-1}y) \mbox{ for all } x, y \in G.
\end{equation}
In other words, $J$ implements the bijective correspondence of Theorem \ref{correspondence}.
\begin{lemma}\label{G-invconv}
   If $\Phi_1, \Phi_2: G/K \times G/K \to \mathbb{R}$ are $G$-invariant kernels, then so is $\Phi_1 \ast \Phi_2$.
\end{lemma}
\begin{proof}
    For all $g \in G$ and for all $x, y \in G$, we have 
    \begin{equation*}
        \begin{split}
            \Phi_1 \ast \Phi_2(gxK,gyK) &= \int_{G/K}\Phi_1(gxK, zK)\Phi_2(zK, gyK)\mbox{d}\mu(zK)\\
            &=\int_{G/K}\Phi_{1}(xK,g^{-1}zK)\Phi_{2}(g^{-1}zK,yK )\mbox{d}\mu (zK)\\
            &=\int_{G/K}\Phi_1(xK,zK)\Phi_2(zK, yK)\mbox{d}\mu(zK)\\
            &= \Phi_{1}\ast \Phi_{2}(xK,yK),
        \end{split}
    \end{equation*}
    where we have used the $G$-invariance of $\Phi_1$ and $\Phi_2$ to go from the first equality to the second, and the $G$-invariance of the measure $\mu$ to go from the second to the third. Hence $\Phi_{1} \ast \Phi_{2}$ is $G$-invariant, as required.
\end{proof}
Let \((G,K)\) be a Gelfand pair, recall that this implies that if $\varphi_{1}, \varphi_2 \in P(G,K)$, then $\varphi_1 \ast \varphi_2 \in P(G,K)$. In particular, since $\varphi_1 \ast \varphi_2$ is positive definite and real-valued, we have
\begin{equation}\label{symmetry}
    \varphi_1 \ast \varphi_2(x)= \varphi_1 \ast \varphi_2(x^{-1}) \mbox{ for all } x \in G.
\end{equation}
\begin{theorem}\label{preservesconvolution}Let \((G, K)\) be a Gelfand pair. Then
 $J(\varphi_1 \ast \varphi_2) = J(\varphi_1) \ast J (\varphi_2)$.
\end{theorem}
\begin{proof}
    Let $\Phi_{i} = J(\varphi_{i})$ for $i = 1, 2$. Then we need to show that 
    \begin{equation*}
        \varphi_1 \ast \varphi_2(y^{-1}x) = \Phi_{1}\ast\Phi_2(yK,xK) \mbox{ for all } x,y \in G.
    \end{equation*}
Now, from \eqref{symmetry}, we have
\begin{equation*}
    \begin{split}
        \varphi_1 \ast \varphi_2(y^{-1}x)&= \varphi_1 \ast \varphi_2(x^{-1}y) \\
        &= \int_{G}\varphi_1(z^{-1}x^{-1}y)\varphi_2(z)\mbox{d}\lambda_{G}(z)\\
         &=\int_{G/K}\int_{K}\varphi_1(k^{-1}z^{-1}x^{-1}y)\varphi_2(zk)\mbox{d}\lambda_{K}(k)\mbox{d}\mu(zK),
    \end{split}
    \end{equation*}
where we have used \eqref{weylintegrationformula} in the last equality. On account of $\varphi_1$ being a positive definite real-valued function, we have that $\varphi_{1}(g) = \varphi_{1}(g^{-1})$ for all $g \in G$. Hence 
\begin{equation*}
     \begin{split}
          \varphi_1 \ast \varphi_2(y^{-1}x) &= \int_{G/K}\int_{K}\varphi_{1}(y^{-1}xzk)\varphi_{2}(zk)\mbox{d}\lambda_{K}(k)\mbox{d}\mu(zK)\\
          &=\int_{G/K}\int_{K}\Phi_1(eK,y^{-1}xzkK)\Phi_2(eK,zkK)\mbox{d}\lambda_{K}(k)\mbox{d}\mu(zK).
     \end{split}
\end{equation*}
Using the $G$-invariance of $\Phi_{1}$, we have $\Phi_1(eK,y^{-1}xzkK) = \Phi_1(x^{-1}yK,zK)$, and hence
\begin{equation*}
    \begin{split}
         \varphi_1 \ast \varphi_2(y^{-1}x)&= \int_{G/K}\int_{K}\Phi_1(x^{-1}yK,zK)\Phi_2(eK,zK)\mbox{d}\lambda_{K}(k)\mbox{d}\mu(zK)\\
         &=\int_{G/K}\int_{K}\Phi_1(x^{-1}yK,zK)\Phi_2(zK,eK)\mbox{d}\lambda_{K}(k)\mbox{d}\mu(zK),
    \end{split}
\end{equation*}
where the last equality follows from the symmetry of the kernel $\Phi_{2}$, that is, for all $x,y \in G$, we have $\Phi_2(xK,yK) = \varphi_2(x^{-1}y) = \varphi_2(y^{-1}x) = \Phi_{2}(yK,xK)$. Finally,
\begin{equation*}
    \begin{split}
        \varphi_1 \ast \varphi_2(y^{-1}x)&= \int_{G/K}\int_{K}\Phi_1(x^{-1}yK,zK)\Phi_2(zK,eK)\mbox{d}\lambda_{K}(k)\mbox{d}\mu(zK)\\
        &= \Phi_1 \ast \Phi_2(x^{-1}yK, eK) = \Phi_1 \ast \Phi_2(yK, xK),
    \end{split}
\end{equation*}
where the last equality follows from the $G$-invariance of $\Phi_1 \ast \Phi_2$ (see Lemma \ref{G-invconv}).

 %
\end{proof}

Note that the proof of Theorem \ref{preservesconvolution} implicitly shows that for $G$-invariant kernels \(\Phi_1, \Phi_2 \in P^{G}(G/K)\), the kernel \(\Phi_1 \ast \Phi_2\) is continuous. It is also easy to see that for all \(s, t >0\) and for all \(\varphi_1, \varphi_2 \in P(G,K)\), \(J(s\varphi_1+t\varphi_2) =  sJ(\varphi_1)+tJ(\varphi_2)\). In other words, \(J\) is additive and positively homogeneous.

\begin{theorem}\label{support}
    Let $G$ be a locally compact group and $K$ a compact subgroup of $G$. Let $\Phi: G/K \times G/K \to \mathbb{C}$ be a $G$-invariant kernel and $\varphi:~G~\ni~g~\mapsto~\Phi(K,gK)~\in~\mathbb{C}$ be the associated $K$-bi-invariant function as in Theorem \ref{correspondence}. Then the following statements hold:
    \begin{itemize}
        \item [(i)] $\Phi(gK,gK) = 1$ for all $g \in G$ if and only if $\varphi(e)=1$.
        \item [(ii)]For all $g \in G$, we have that $\Phi(K,gK) \le 0$ if and only if $\varphi(g) \le 0$, and $\Phi(K,gK) \ge 0$ if and only if $\varphi(g) \ge 0$.
        \item[(iii)] We have \begin{equation*}
        \int_{G}\varphi(g)\textup{d}\lambda_{G}(g) = \int_{G/K}\Phi(K,gK)\textup{d}\mu(gK)
    \end{equation*}
    for all $\varphi \in C_{c}^{K}(G)$, where $\mu$ denotes the invariant measure on $G/K$ satisfying \eqref{weylintegrationformula}.
    \end{itemize}
\end{theorem}
\begin{proof}
By definition of $\varphi$ and the $G$-invariance of $\Phi$, we have $\varphi(e) = \Phi(K,K) = \Phi(xK,xK)$ for all $x \in G$. Thus, $\Phi(xK,xK) = 1$ for all $x \in G$ if and only if $\varphi(e) =1$. This proves (i). \par
For (ii), it is enough to observe that  $\varphi(g) = \Phi(K,gK)$ for all $g \in G$. \par
Finally, using \eqref{weylintegrationformula} and $\lambda_{K}(K)=1$, we have
\begin{equation*}
\begin{split}
    \int_{G}\varphi(g)\textup{d}\lambda_{G}(g) &=  \int_{G/K}\int_{K}\varphi(gk)\textup{d}\lambda_{K}(k)\textup{d}\mu(gK)\\
    &= \int_{G/K}\int_{K}\Phi(K,gkK)\textup{d}\lambda_{K}(k)\textup{d}\mu(gK)\\
    &=\int_{G/K}\int_{K}\Phi(K,gK)\textup{d}\lambda_{K}(k)\textup{d}\mu(gK)\\
    &= \int_{G/K}\Phi(K,gK)\textup{d}\mu(gK),
\end{split}
\end{equation*}
proving (iii).
\end{proof}
\section{Delsarte-type problem on extremal problem}
In this section, we consider the homogeneous space $G/K$, where $G$ is a locally compact group and $K$ a compact subgroup. 
\subsection{Delsarte problem for invariant kernels}
For a \(K\)-bi-invariant symmetric identity neighbourhood $U \subset G$, and a \(K\)-bi-invariant symmetric subset $V \subset G$, define the collection 

\begin{equation}\label{delsarteclass}
    \mathcal{F}_{G/K}(U,V):=\Bigl \{\Phi \in P^{G}_{1}(G/K):\supp \Phi_{+} \subset U, \supp \Phi_{-} \subset V \Bigr \}.
\end{equation}
We consider the Delsarte-type extremal problem of computing
\begin{equation}\label{sphericaldelsarteconstant}
   \mathcal{C}_{G/K}(U,V):= \sup_{\Phi \in \mathcal{F}_{G/K}(U,V)} \int_{G/K}\Phi(xK,yK)\textup{d}\mu(yK).
\end{equation}
Our motivation for studying this Delsarte-type extremal problem comes from the Tur\'an problem for a spherical cap \cite[Problem 4]{gneitingsupplement}. Let us formulate the Tur\'an problem in the setting of homogeneous spaces. Let $U \subset G$ be a symmetric $K$-bi-invariant compact neighbourhood of the identity in $G$, and put $\mathcal{F}_{G/K}(U):=\mathcal{F}_{G/K}(U,U)$. In other words \(\mathcal{F}_{G/K}(U)\) consists of continuous positive definite \(G\)-invariant kernels \(\Phi: G/K \times G/K \to \mathbb{R}\) satisfying \(\Phi(xK,xK) = 1\) for all \(x \in G\), and having \(\supp \Phi \subset U \). The Tur\'an problem asks for
\begin{equation}\label{invturan}
     \mathcal{T}_{G/K}(U) := \sup_{\Phi \in \mathcal{F}_{G/K}(U)} \int_{G/K}\Phi(K,gK)\mbox{d}\mu(gK).
\end{equation}
Note that for any $x \in G$, we have
\begin{equation}\label{independence}
\begin{split}
     \int_{G/K}f(xK, gK)\mbox{d}\mu(gK) &=  \int_{G/K}f(K, x^{-1}gK)\mbox{d}\mu(gK)\\
     &= \int_{G/K}f(K, gK)\mbox{d}\mu (xgK) \\
     &= \int_{G/K}f(K,gK)\mbox{d}\mu(gK).
\end{split} 
\end{equation}
The first equality in \eqref{independence} is justified by the $G$-invariance of $f$, and the third by the $G$-invariance of the measure $\mu$. This shows that the choice of $x$ in the integrands \eqref{independence} does not matter for the sake of calculating the integral. Nevertheless, since $K$ is the stability subgroup of the coset $eK = K \in G/K$, it makes sense to pick the coset $K$ as a basepoint in \eqref{invturan}.\par
Let us specialise the Tur\'an problem \eqref{invturan} to the sphere. Let $x = (0, \ldots,0,1) \in \mathbb{S}^d$ be the north pole and $K$ be the subgroup of $\mathrm{SO}(d+1)$ consisting of rotations that fix $x$. Recall that $K$ is isomorphic to $\mathrm{SO}(d)$, and that $\mathbb{S}^d \cong \mathrm{SO}(d+1)/\mathrm{SO}(d)$. For $0 < c \le \pi$, let $U$ be the identity neighbourhood in $\mathrm{SO}(d+1)$ consisting of rotations $T \in \mathrm{SO}(d+1)$ such that $\theta(Tx,x) \le c$, where $\theta(Tx,x)$ is the great-circle distance between $x$ and $Tx$. Suppose that $T \in U$, so that $\theta(Tx,x) \le c$. Then, for any $k, k' \in K$, we have
\begin{equation*}
    \theta(kTk'x,x) = \theta(Tk'x,k^{-1}x) = \theta(Tx,x) \le c.
\end{equation*}
Therefore, $kTk' \in U$. In other words, $U$ is $K$-bi-invariant. Clearly, the image $UK = \bigcup_{u \in U}uK$ of $U$ in $\mathbb{S}^d$ under the canonical projection $p: \mathrm{SO}(d+1)/\mathrm{SO}(d) \to \mathbb{S}^d$ is the spherical cap of radius $c$ centred at $x$.  It follows that \eqref{invturan} specialises to \eqref{spherturan}, showing that we obtain a true generalisation of the spherical Tur\'an problem \eqref{spherturan}.

Even though the spherical Tur\'an problem \eqref{spherturan} is covered in \eqref{invturan}, we study the more general extremal problem \eqref{sphericaldelsarteconstant}. In any case, we have
\begin{equation*}
    \mathcal{T}_{G/K}(U)  = \mathcal{C}_{G/K}(U,U).
\end{equation*}

\subsection{Delsarte problem for \(K\)-bi-invariant functions on groups}

Motivated by Theorem \ref{support}, we introduce the following Delsarte-type problem for continuous positive definite $K$-bi-invariant functions. Let $G$ be a locally compact group, $K$ a compact subgroup of $G$. Let $U$ be a $K$-bi-invariant symmetric identity neighbourhood\footnote{Such a $K$-bi-invariant neighbourhood exists by first taking $U_{1}=KWK$ where $W \subset G$ is some identity neighbourhood, and then letting \(U = U_{1} \cap U{_1}^{-1}\).}, and $V$ a $K$-bi-invariant symmetric subset of $G$. Then denote by $\mathcal{F}_{G}^{K}(U,V)$ the collection of continuous positive definite $K$-bi-invariant functions $\varphi: G \to \mathbb{R}$ satisfying $\varphi(e)=1$ and $\supp\varphi_{+} \subset U$ and $\supp \varphi_{-} \subset V$. So, we have the class

\begin{equation*}
    \mathcal{F}_{G}^{K}(U,V) := \Bigl \{\varphi \in P_{1}(G,K): \supp\varphi_{+} \subset U,\supp \varphi_{-} \subset V  \Bigr \},
\end{equation*}
and the respective version of the Delsarte problem asks for:
\begin{equation}\label{invariant-turan}
     \mathcal{C}_{G}^{K}(U,V) := \sup_{\varphi \in \mathcal{F}_{G}^{K}(U,V)} \int_{G}\varphi(g)\textup{d}\lambda_{G}(g).
\end{equation}
Putting $\mathcal{F}^{K}_{G}(U,U) := \mathcal{F}^{K}_{G}(U)$ and $\mathcal{C}^{K}_{G}(U,U):=\mathcal{T}^{K}_{G}(U) $, the Tur\'an problem for $K$-bi-invariant functions asks for
\begin{equation}
    \mathcal{T}_{G}^{K}(U):= \sup_{\mathcal{F}_{G}^{K}(U)}\int_{G}f(g)\mbox{d}\lambda_{G}(g).
\end{equation}
By Theorem \ref{support}, we have the following theorem connecting the Delsarte problem in \eqref{invariant-turan} above to the one in \eqref{sphericaldelsarteconstant}.
\begin{theorem}\label{turanconstants}
    Let $G$ be a locally compact group, $K$ a compact subgroup of $G$, $U$ a $K$-bi-invariant symmetric neighbourhood of the identity in $G$, and $V$ a $K$-bi-invariant symmetric subset of $G$. Then
    \begin{equation*}
    \mathcal{C}_{G/K}(U,V)  =  \mathcal{C}_{G}^{K}(U,V).
\end{equation*}
\end{theorem}
\begin{corollary}
    Let $G$ be a locally compact group, $K$ a compact subgroup of $G$, and $U$ a $K$-bi-invariant symmetric neighbourhood of the identity in $G$. Then 
    \begin{equation*}
        \mathcal{T}_{G/K}(U)=\mathcal{T}_{G}^{K}(U).
    \end{equation*}
\end{corollary}
We have shown that the Delsarte problem on the homogeneous space $G/K$ for $G$-invariant functions is equivalent to the Delsarte problem for $K$-bi-invariant functions on $G$.

Theorems \ref{correspondence}, \ref{support}, and \ref{turanconstants} are interesting since we can use information about $\mathcal{F}_{G}^{K}(U,V)$ and $\mathcal{C}_{G}^{K}(U,V)$ to say something about $\mathcal{F}_{G/K}(U,V)$ and $\mathcal{C}_{G/K}(U,V)$, and the other way around. For example, the question of whether $\mathcal{F}_{G/K}(U,V)$ is non-empty is equivalent to the question of whether $\mathcal{F}_{G}^{K}(U,V)$ is non-empty. Let us answer the latter, noting that it is enough to construct a continuous positive definite $K$-bi-invariant function supported in $U$.\par
Firstly, observe that since $U$ is $K$-bi-invariant, we have $KUK = U$. So, if $f$ is a continuous positive definite function supported in $U$, it follows from Lemma \ref{Ksymmetrisation}(i) that $f^{K}$ is a continuous positive definite $K$-bi-invariant function supported in $KUK = U$. Therefore, it suffices to find a continuous positive definite function $f$ supported in $U$. For this, we take a compact symmetric identity neighbourhood $W \subset G$ such that $W^{2} = WW \subset U$. Let $\mathbf{1}_{W}$ be the characteristic function of $W$. Then the self-convolution $f =\mathbf{1}_{W} \ast \mathbf{1}_{W}$ is a continuous positive definite function by Proposition \ref{autocorrelation}, supported in $W^2 \subset U$. We can renormalise $f$ so that $f(e) = 1$, and hence $f \in \mathcal{F}_{G}^{K}(U,V)$. Now, since the collection $\mathcal{F}_{G}^{K}(U,V)$ is non-empty, it follows from Theorem \ref{support} that $\mathcal{F}_{G/K}(U,V)$ is non-empty. \par

Let us finish this section by showing how we recover Delsarte-type problems on Abelian groups. 

Let $G$ be an LCA group. Then $G$ is a homogeneous $G$-space with $G$ acting on itself by translations. The stability subgroup of any point is the trivial subgroup $K=\{0\}$. In this case, every function on $G$ is obviously $K$-bi-invariant. Therefore,  $\mathcal{F}_{G}(U,V) = \mathcal{F}_{G}^{K}(U,V)$, and so we recover the Delsarte-type extremal problem on LCA groups considered in \cite{berdrevram}. Also, every $G$-invariant positive definite function is of the form $\Phi(xK,yK) = \varphi(y-x)$ for all $x,y \in G$. Hence, this Delsarte-type problem on LCA groups is a special case of the Delsarte-type problem on homogeneous spaces.\par

\subsection{Existence of extremisers} An application of the bijective correspondence of Theorem \ref{correspondence} is a proof of the existence of an extremal kernel for the Delsarte problem for $G$-invariant kernels on a homogeneous space. The idea is to solve the problem of the existence of an extremal function for \(K\)-bi-invariant functions on \(G\) and use Theorem \ref{correspondence} to conclude that an extremal kernel for the extremal problem on \(G/K\).

We prove the existence of an extremal function for homogeneous spaces \(G/K\) for which the group \(G\) is amenable in the sense of \cite[Definition 4.1]{pier}. What we need is the fact that any positive definite function \(f: G \to \mathbb{R}\) on an amenable group \(G\) has a non-negative integral \cite[Theorem 8.9 (ii')]{pier}. The method we use here closely follows that of \cite{ramabulana, berdrevram}; see also the earlier work \cite{marcell-zsuzsa}. 

Let us start with some lemmas.
\begin{lemma}\label{Ksymmetrisation} Let $G$ be a locally compact group, $K$ a compact subgroup of $G$, and $U$ a $K$-bi-invariant subset of $G$. Let $f \in C_{c}(G)$, then the following statements hold. 
\begin{itemize}
    \item [(i)] If $\supp f_{\pm} \subset U$, then $\supp f^{K}_{\pm} \subset U$.
    \item [(ii)] If $f$ is positive definite, then so is $f^K$.
    \item [(iii)] We have the equality of the integrals
    \begin{equation*}
        \int_{G}f^{K}(g)\textup{d}\lambda_{G}(g) =   \int_{G}f(g)\textup{d}\lambda_{G}(g).
    \end{equation*}
\end{itemize}
\end{lemma}
\begin{proof}
    For (i), let us suppose that $\supp f_{+} \subset U$. Let $g \in G$ be such that $g \notin U$. We claim that $kgk' \notin U$ for all $k, k' \in K$. Indeed, if $kgk' \in U$, then $g \in KUK = U$. 

    Now, we have that $f(kgk') \le 0$ for all $g \notin U$ and all $k, k' \in K$. Therefore,
    \begin{equation*}
        f^{K}(g)=\int_{K}\int_{K}f(kgk')\mbox{d}\lambda_{K}(k)\mbox{d}\lambda_{K}(k') \le 0 \mbox{ for all } g \notin U,
    \end{equation*}
    and hence $\supp f^{K}_{+} \subset U$. An analogous argument shows that $\supp f^{K}_{-} \subset U$ whenever $\supp f_{-} \subset U$.
    
    For (ii), we have 
    \begin{equation*}
        \begin{split}
            \sum_{i=1}^{n}\sum_{j=1}^{n}f^{K}(g_i^{-1}g_j)c_i\overline{c_j} &=  \sum_{i=1}^{n}\sum_{j=1}^{n}\int_{K}\int_{K}f(kg_{i}^{-1}g_{j}k')c_i\overline{c_j}\textup{d}\lambda_{K}(k)\textup{d}\lambda_{K}(k')\\
            &= \int_{K}\int_{K} \sum_{i=1}^{n}\sum_{j=1}^{n}f((g_{i}k^{-1})^{-1}(g_jk'))c_i\overline{c_j}\textup{d}\lambda_{K}(k)\textup{d}\lambda_{K}(k') \ge 0.
        \end{split}
    \end{equation*}
    Therefore, $f^{K}$ is positive definite. 
    
    For (iii), let us use Fubini's Theorem to calculate the integral of $f^{K}$. We have
\begin{equation*}\label{tripleintegral}
    \begin{split}
    \int_{G}f^{K}(g)\textup{d}\lambda_{G}(g) &= \int_{G} \int_{K}\int_{K}f(kgk')\textup{d}\lambda_{K}(k)\textup{d}\lambda_{K}(k')\textup{d}\lambda_{G}(g)\\
    &= \int_{K} \int_{K}\int_{G}f(kgk')\textup{d}\lambda_{G}(g)\textup{d}\lambda_{K}(k)\textup{d}\lambda_{K}(k')\\
    &= \int_{K} \int_{K}\int_{G}f(g)\textup{d}\lambda_{G}(g)\textup{d}\lambda_{K}(k)\textup{d}\lambda_{K}(k')\\
    &=\int_{G}f(g)\textup{d}\lambda_{G}(g),
    \end{split}
\end{equation*}
where we have used the left invariance of the Haar measure $\lambda_{G}$ and the fact that $\Delta_{G}(k) = 1$ for all $k \in K$ (see \eqref{compactmodularfuction}) to go from the second equality to the third one, and the normalisation $\int_{K}1\textup{d}\lambda_{K} =1$ to get the last equality.  
\end{proof}

\begin{lemma}[Mazur's Lemma, {\cite[Corollary 3.8, Exercise 3.4]{brezis}}] \label{mazur} Let $E$ be a Banach space and let $(x_{n})_{n \in \mathbb{N}}$ be a sequence in $E$ converging to $x \in E$ weakly. Then there exists a sequence $(y_{n})_{n \in \mathbb{N}}$ in $E$ such that
	\begin{equation*}
	y_{n} \in \mathrm{conv}\left(\bigcup_{i=n}^{\infty}\{x_{i}\}\right) \textup{ for all } n \in \mathbb{N}
	\end{equation*}
and $(y_n)_{n \in \mathbb{N}}$ converges strongly to $x$ in $E$, where for a subset $A \subset E$, $\mathrm{conv}(A)$ refers to the convex hull of $A$, and strong convergence refers to convergence with respect to the norm topology on $E$.
\end{lemma}

We prove the existence of an extremal function for $\sigma$-compact groups first, after which we shall extend it to groups that are not necessarily $\sigma$-compact.
\begin{theorem}\label{Ksigmacompactcase}
    Let $G$ be an amenable $\sigma$-compact locally compact group, $K$ a compact subgroup of $G$, $U$ a closed $K$-bi-invariant symmetric identity neighbourhood in \(G\) having finite Haar measure, and \(V\) a \(K\)-bi-invariant symmetric subset of \(G\). Then there exists an extremal function for the Delsarte constant $\mathcal{C}_{G}^{K}(U,V)$.
\end{theorem}
\begin{proof}
Note that for all $f \in  \mathcal{F}_{G}^{K}(U,V)$, we have $|f(g)|\le 1$ for all $g \in G$. Furthermore, since $G$ is amenable, $f$ has non-negative integral \cite[Theorem 8.9]{pier}. The non-negativity of the integral of $f$ implies the estimate $$\int_G f_{-}\mbox{d}\lambda_G \le \int_{G}f_{+}\mbox{d}\lambda_{G} \le \lambda_{G}(U) < \infty.$$Hence $\|f\|_{L^{1}(G)} \le 2 \lambda_{G}(U)$, which in turn implies that $\|f\|_{L^{2}(G)} \le \sqrt{2\lambda_{G}(U)}$ for all $f\in~\mathcal{F}_{G}^{K}(U,V)$. Let $r: = \sqrt{2\lambda_{G}(U)}$. Then, $\mathcal{F}_{G}^{K}(U,V)$ is a bounded subset of $L^{2}(G)$, and as such it belongs to the strong closure $\overline{B_{r}(0)}$ of the ball $B_{r}(0)$ of radius $r$ in $L^2(G)$ centred at 0. Being a closed and bounded subset of $L^2(G)$, $\overline{B_{r}(0)}$ is weakly sequentially compact according to \cite[Theorem 3.18]{brezis}.\par
    Now let $(f_{n})_{n \in \mathbb{N}}$ be a $\mathcal{C}^{K}_{G}(U, V)$-extremal sequence in $\mathcal{F}_{G}^{K}(U,V)$, that is, $$\lim_{n \to \infty} \int_{G}f_n(g)\textup{d}\lambda_{G}(g) = \mathcal{C}^{K}_{G}(U, V).$$ By the weak sequential compactness of $\overline{B_{r}(0)}$, there is a subsequence of $(f_{n})_{n \in \mathbb{N}}$ that converges weakly in $L^2(G)$ to some $f \in \overline{B_{r}(0)}$. By replacing $(f_{n})_{n \in \mathbb{N}}$ with this subsequence, assume that $(f_{n})_{n \in \mathbb{N}}$ converges weakly in $L^2(G)$ to $f$. Since $\mathcal{F}_{G}^{K}(U,V)$ is a convex subset of $L^2(G)$, Lemma \ref{mazur} implies that there is a sequence $(h_{n})_{n \in \mathbb{N}}$ in $\mathcal{F}_{G}^{K}(U,V)$ chosen such that $h_{n} \in \mathrm{conv} (\bigcup _{i =n}^{\infty}{f_{i}})$ for all $n \in \mathbb{N}$, and converging to $f$ strongly in $L^2(G)$. We claim that the latter sequence is $\mathcal{C}^{K}_{G}(U, V)$-extremal. First note that by the extremality of $(f_{n})_{n \in \mathbb{N}}$, we have that for any $\varepsilon > 0$ there exists $N \in \mathbb{N}$ such that for all $n > N$ it holds that
	\begin{equation*}
	\left |\int_{G}f_{n}(g)\textup{d}\lambda_{G}(g) - \mathcal{C}^{K}_{G}(U, V) \right| < \varepsilon.
	\end{equation*}
Now, write $h_{n} = \sum_{i =n}^{\infty}c_{i}^{n}f_{i}$ where $c_{i}^{n} \ge 0$, $\sum_{i = n}^{\infty}c_{i}^{n} = 1$, and $c_{i}^{n} = 0$ for all but finitely many $i \ge n$. Then for any $n \ge N$, we have
	\begin{equation*}
	\left | \int_{G}h_{n}(g)\textup{d}\lambda_{G}(g) - \mathcal{C}^{K}_{G}(U, V) \right| \le \sum_{i =n}^{\infty}c_{i}^{n}\left|\int_{G}f_{i}(g)\textup{d}\lambda_{G}(g) - \mathcal{C}^{K}_{G}(U, V) \right| < \sum_{i = n}^\infty c_{i}^{n}\varepsilon = \varepsilon.
	\end{equation*}
Therefore $(h_{n})_{n \in \mathbb{N}}$ is $\mathcal{C}^{K}_{G}(U, V)$-extremal. So, replacing, if necessary, the sequence $(f_{n})_{n \in \mathbb{N}}$ by this new sequence, assume that $(f_{n})_{n \in \mathbb{N}}$ converges to $f$ strongly in $L^2(G)$.\par
By passing to a pointwise almost everywhere convergent subsequence, assume that $(f_{n})_{n \in \mathbb{N}}$ converges to $f$ pointwise almost everywhere. The fact that for all $n \in \mathbb{N}$, $|f_{n}(g)|\le~1$ for all $g \in G$ and $(f_{n})_{n \in \mathbb{N}}$ converges almost everywhere to $f$ implies that $f$ is bounded in $L^{\infty}(G)$ by 1. Since $(f_{n})_{n \in \mathbb{N}}$ converges weakly in $L^2(G)$ to $f$, we have
	$$
	\int_{G}(h^{\ast} \ast h)(g)f(g) \textup{d}\lambda_{G}(g) = \lim_{n \to \infty} \int_{G}(h^{\ast}\ast h)(g)f_{n}(g)(g)\textup{d}\lambda_{G}(g) \ge 0,
	$$
	for all compactly supported continuous complex-valued functions $h$. This shows that $f$ is integrally positive definite. The function $f$ being integrally positive definite and $G$ being $\sigma$-compact implies that $f$ agrees almost everywhere with a continuous positive definite function. Thus, by correcting $f$ on a set of Haar measure zero, thus not changing the value of its integral, assume that $f$ is a continuous positive definite function.\par
The fact that $\supp f_{+} \subset U$ and $\supp f_{-} \subset V$ follows from the following facts: that $f$ is continuous, $(f_n)_{n \in \mathbb{N}}$ converges to $f$ almost everywhere, that $\supp({f}_n)_{+} \subset U$, $\supp({f}_n)_{-} \subset V$ and that $U$ and $V$ are closed.

Let us consider the projection $P^{K}(f) =f^{K}$ of $f$. By Lemma \ref{Ksymmetrisation}, $f^K$ is a continuous positive definite $K$-bi-invariant function with $\supp f^{K}_{+}  \subset U$, $\supp f^{K}_{-} \subset V$, and having total integral equal to that of $f$. Therefore, by replacing $f$ by $f^{K}$, let us assume that $f$ is $K$-bi-invariant.\par
Since $\supp f_{+} \subset U$ we have 

       \begin{equation}\label{positivitybound}
        \int_{G}f_{+}(g)\textup{d}\lambda_{G}(g) = \int_{U}f_{+}(g)\textup{d}\lambda_{G}(g) \le \lambda_{G}(U) < \infty.
    \end{equation}
    On the other hand, by Fatou's lemma, 

    \begin{equation}\label{negativitybound}
        \begin{split}
            \int_{G}f_{-}(g)\textup{d}\lambda_{G}(g) & \le \liminf_{n \to \infty} \int_{G}(f_{n})_{-}(g)\textup{d}\lambda_{G}(g)\\ 
            &\le \liminf_{n \to \infty}\int_{G}(f_{n})_{+}(g)\textup{d}\lambda_{G}(g) \\
            & \le \lambda_{G}(U) < \infty,    
        \end{split}
    \end{equation}
    where the second inequality of \eqref{negativitybound} is obtained using that $f_{n}$ being positive definite and integrable implies that
    \begin{equation*}
       0 \le \int_{G}f_{n}(g)\mbox{d}\lambda_{G}(g) = \int_{G}(f_{n})_{+}(g)\mbox{d}\lambda_{G}(g) - \int_{G}(f_{n})_{-}(g)\mbox{d}\lambda_{G}(g).
    \end{equation*} 
     Together, \eqref{positivitybound} and \eqref{negativitybound} imply that 
\begin{equation*}
      \int_{G}|f(g)|\textup{d}\lambda_{G}(g) =   \int_{G}f_{+}(g)\textup{d}\lambda_{G}(g) +  \int_{G}f_{-}(g)\textup{d}\lambda_{G}(g) < \infty.
\end{equation*}
Therefore, $f \in L^1(G)$. Next, we show that $\int_{G}f\mbox{d}\lambda_{G} = \mathcal{C}^{K}_{G}(U,V)$. To that end, note that the sequence $((f_{n})_{+})_{n \in \mathbb{N}}$ of the positive parts of the $f_{n}$'s converges pointwise almost everywhere to $f_{+}$, and that $(f_{n})_{+} \le \mathbf{1}_{U}$. Therefore, by Lebesgue's dominated convergence theorem and $\lambda_{G}(U) < \infty$, we have
  \begin{equation}\label{eq10}
     \int_{G}f_{+}(g)\mbox{d}\lambda_{G}(g) =  \lim_{n \to \infty} \int_{G}(f_{n})_{+}(g)\mbox{d}\lambda_{G}(g).
  \end{equation}
  Recall that from  \eqref{negativitybound} we have the inequality
  \begin{equation}\label{eq9}
            \int_{G}f_{-}(g)\lambda_{G}(g) \le \liminf_{n \to \infty}\int_{G}(f_{n})_{-}(g)\mbox{d}\lambda_{G}(g).
 \end{equation}
  Now, subtracting \eqref{eq9} from \eqref{eq10} we have
  \begin{equation*}
\begin{split}
    \int_G f(g)  \text{d}\lambda_G(g) &\ge  \lim_{n \to \infty} \int_{G}(f_n)_{+}(g)\text{d}\lambda_G(g) - \liminf_{n \to \infty}{\int_{G}(f_n)_{-}(g)} \text{d}\lambda_G(g) \\
    &= \limsup_{n \to \infty}{\int_{G}\left((f_n)_{+}(g) - (f_n)_{-}(g)\right) \text{d}\lambda_G(g)} \\
    &= \lim_{n \to \infty}\int_{G}f_n (g) \mbox{d}\lambda _{G}(g) = \mathcal{C}^{K}_{G}(U,V).
\end{split}
\end{equation*}

Finally, we show that $f(e)=1$. To that end, note that $f(e) \le 1$, otherwise, we could find some $f_n$ and some $g \in U$ such that $f_n(g) > 1$, which would contradict $f_n(e)\le~1$. Also, $f(e) \ne 0$, otherwise $f$ would be identically zero, contradicting the fact that $\int_{G}f(g)\textup{d}\lambda_{G}(g)=\mathcal{C}_{G}^{K}(U,V) \ne 0$. So, in order to show that $f(e) = 1$, we assume, for a contradiction, that $0<f(e) <1$. Then the function $h := \frac{f}{f(e)}$ is in $\mathcal{F}_{G}^{K}(U, V)$ and $\int_{G}h(g)\textup{d}\lambda_{G}(g) > \mathcal{C}_{G}^{K}(U,V)$, contradicting the optimality of $\mathcal{C}_{G}^{K}(U,V)$. It follows that $f(e) = 1$, and hence $f \in \mathcal{F}_{G}^{K}(U, V)$, as required.
\end{proof}
Let us now extend Theorem \ref{Ksigmacompactcase} to the case where $G$ is not a $\sigma$-compact group.

 For a function $\varphi: H \to \mathbb{C}$ defined on the subgroup $H$ of $G$, its trivial extension is the function $\widetilde{\varphi}: G \to \mathbb{C}$ defined on $G$, and given by
\begin{equation*}
\widetilde{\varphi}(g) := \begin{cases}
\varphi(g) & \text{if } g \in H, \\
0  & \text{if } g \in G\setminus H.
\end{cases}
\end{equation*}

The following lemma is obvious.
\begin{lemma}\label{restrictextendKbi}
    Let $G$ be a locally compact group, $K$ a compact subgroup of $G$, and $H$ an open subgroup of $G$ containing $K$. Then the following statements hold:
    \begin{itemize}
        \item [(i)] For a $K$-bi-invariant function $f: G \to \mathbb{C}$, its restriction $f|_H: H \to \mathbb{C}$ to the subgroup $H$ is $K$-bi-invariant.
        \item [(ii)]  For a $K$-bi-invariant function $f: H \to \mathbb{C}$, its trivial extension $\widetilde{f}: G \to \mathbb{C}$ is $K$-bi-invariant.
    \end{itemize} 
\end{lemma}

Using Lemma \ref{restrictextendKbi}, we have the following Lemma.

\begin{lemma}\label{extensionlemma}
	Let $G$ be a locally compact group, $K$ a compact subgroup of $G$, and $H$ an open subgroup of $G$ containing $K$. Let $U$ be a $K$-bi-invariant symmetric neighbourhood of identity in $G$, and $V$ be $K$-bi-invariant symmetric subset of $G$. If a function $\varphi : H \to \mathbb{R}$ is in $\mathcal{F}^{K}_{H}(U \cap H, V \cap H)$, then its trivial extension $\widetilde{\varphi}: G \to \mathbb{R}$ is in $\mathcal{F}^{K}_{G}(U, V)$.
\end{lemma}
\begin{proof}
    Suppose that $\varphi: H \to \mathbb{R}$ is in $\mathcal{F}^{K}_{H}(U \cap H, V \cap H)$. This means that $\varphi$ is continuous, positive definite, has $\supp \varphi_+ \subset U$, $\supp \varphi_{-} \subset V$, and satisfies $\varphi(0) = 1$. Clearly, its trivial extension satisfies $\supp \widetilde{\varphi}_{+} \subset U$ and $\supp \widetilde{\varphi}_{-} \subset V$, and $\widetilde{\varphi}(0) = 1$. To see that $\widetilde{\varphi}$ is continuous, note that $H$, being an open subgroup, is also closed, and hence its complement $G \setminus H$ is open. Now,  $\widetilde{\varphi}$ is continuous at all $g \in H$ because it coincides with the continuous function $\varphi: H \to \mathbb{R}$ on the open set $H$. It is continuous at all $g \in G \setminus H$ because it is identically 0 on the open set $G \setminus H$.  It is positive definite by  \cite[Theorem 32.43(a)]{hewittross2}. 
\end{proof} \par 
We shall also make use of the following obvious lemma.
\begin{lemma}\label{restrictionlemma}
	Let $G$ be a locally compact group and $K$ a compact subgroup of $G$, and $H$ an open subgroup $G$ containing $K$. Let $U$ be a $K$-bi-invariant symmetric neighbourhood of identity in $G$ with $U \subset H$, and $V$ be a $K$-bi-invariant symmetric subset of $G$. If a function $\varphi: G \to \mathbb{R}$ is in $\mathcal{F}^{K}_{G}(U, V)$, then its restriction $\varphi |_{H}: H \to \mathbb{R}$ to the subgroup $H$ is in $\mathcal{F}^{K}_{H}(U, V \cap H)$.
\end{lemma}

\begin{lemma}[{\cite[Lemma 5]{berdrevram}}]\label{sigma-compact-lemma}
A Haar measurable subset $U$ of a LCA group $G$ is $\sigma$-finite if and only if it is contained in an open $\sigma$-compact subgroup.    
\end{lemma}

\begin{lemma}[Compare {\cite[Theorem 3]{ramabulana}, \cite[Theorem 6]{berdrevram}}]\label{Kreduction}Let $G$ be a locally compact group, $K$ a compact subgroup of $G$, $U$ a $\sigma$-finite $K$-bi-invariant symmetric neighbourhood of $G$, and $V$ be a $K$-bi-invariant symmetric subset of $G$. Then there exists an open $\sigma$-compact subgroup $H$ of $G$, containing $K$ and $U$, such that
\begin{equation*}
    \mathcal{C}_{G}^{K}(U,V) = \mathcal{C}_{H}^{K}(U,V).
\end{equation*}
\end{lemma}
\begin{proof}
    On account of Lemma \ref{sigma-compact-lemma}, let $H$ be an open $\sigma$-compact subgroup of $G$ containing the union $U \cup K$ of $U$ and the subgroup $K$. Equip $H$ with Haar measure $\lambda_{H}: = \lambda_{G}|_H$, i.e., the Haar measure of $G$ restricted to $H$. Let $(f_n)_{n \in \mathbb{N}}$ be a $\mathcal{C}_{G}^{K}(U,V)$-extremal sequence in $\mathcal{F}_{G}^{K}(U, V)$ such that
    \begin{equation*}
        \int_{G}f_n(g)\textup{d}\lambda_{G}(g) \ge \mathcal{C}_{G}^{K}(U,V) - \frac{1}{n} \textup{ for all } n \in \mathbb{N}.
    \end{equation*}
Let $f_n|_H: H \to \mathbb{R}$ denote the restriction of $f_n$ to $H$. Observe that for all $n \in \mathbb{N}$, we have
\begin{equation}\label{splitintegral2}
    \int_{G}f_n(g)\textup{d}\lambda_{G}(g) = \int_{H}f_n(g)\textup{d}\lambda_{G}(g) + \int_{G \setminus H}f_n(g)\textup{d}\lambda_{G}(g).
\end{equation}
Since $\supp{f_n} \subset U$ is contained in $H$, we have $f_n(g) \le 0$ for all $ g \in G \setminus H$. Therefore,
\begin{equation}\label{outsideH}
    \int_{G \setminus H}f_{n}(g)\textup{d}\lambda_{G}(g) \le 0.
\end{equation}
From \eqref{splitintegral2} and \eqref{outsideH}, it follows that 
\begin{equation}\label{restrictioninequality}
    \int_{H}f_n|_{H}(g)\textup{d}\lambda_{H}(g) = \int_{H}f_n(g)\textup{d}\lambda_{G}(g) \ge \int_{G}f_n(g)\textup{d}\lambda_{G}(g).
\end{equation}
Now, by definition of $\mathcal{C}_{H}^{K}(U, V)$ and inequality \eqref{restrictioninequality}, we have
\begin{equation*}
    \mathcal{C}_{H}^{K}(U, V) \ge \int_{H}f_n|_{H}(g)\textup{d}\lambda_{H}(g) \ge \int_{G}f_n(g)\textup{d}\lambda_{G}(g) \ge \mathcal{C}_{G}^{K}(U,V) - \frac{1}{n}.
\end{equation*}
Thus $\mathcal{C}_{H}^{K}(U,V) \ge \mathcal{C}_{G}^{K}(U,V)$. \par
For the reverse inequality, let $(h_n)_{n \in \mathbb{N}}$ be a  $\mathcal{C}_{H}^{K}(U,V)$-extremal sequence in $\mathcal{F}_{H}^{K}(U,V)$ such that
\begin{equation*}
    \int_{H}h_n(g)\textup{d}\lambda_{H}(g) \ge \mathcal{C}_{H}^{K}(U,V) - \frac{1}{n} \textup{ for all } n \in \mathbb{N}.
\end{equation*}
The sequence $(\widetilde{h_n})_{n \in \mathbb{N}}$ of trivial extensions is in $\mathcal{F}_{G}^{K}(U,V)$, by Lemma \ref{restrictextendKbi}. Since $\widetilde{h_{n}}$ vanishes outside of $H$, we have that for all $n \in \mathbb{N}$,
			
			\begin{equation*}
				\mathcal{C}_{G}^{K}(U,V) \ge \int_{G}\widetilde{h_{n}}(g)\textup{d}\lambda_{G}(g) = \int_{H}h_{n}(g)\textup{d}\lambda_{H}(g) \ge \mathcal{C}_{H}^{K}(U,V) - \frac{1}{n}.
			\end{equation*}
			
			Thus, $\mathcal{C}_{G}^{K}(U,V) \ge \mathcal{T}_{H}^{K}(U)$, and hence 

   \begin{equation*}
        \mathcal{C}_{G}^{K}(U,V) = \mathcal{C}_{H}^{K}(U,V),
   \end{equation*}
   as required.
\end{proof}
\begin{theorem}\label{generalKtheorem}Let $G$ be an amenable locally compact group, $K$ a compact subgroup of $G$, $U$ a closed $K$-bi-invariant symmetric identity neighbourhood having finite Haar measure, and $V$ a closed $K$-bi-invariant symmetric subset of $G$. Then there exists an extremal function for the Delsarte constant $\mathcal{C}_{G}^{K}(U,V)$. 
\end{theorem}
\begin{proof}
On account of Theorem \ref{Kreduction}, let $H$ be an open $\sigma$-compact subgroup of $G$ containing $K$ and $U$ such that $\mathcal{C}_{H}^{K}(U,V) = \mathcal{C}_{G}^{K}(U,V)$. By Theorem \ref{Ksigmacompactcase}, there exists an extremal function $f \in \mathcal{F}_{H}^{K}(U, V)$ such that 
\begin{equation*}
    \int_{H}f(g)\mbox{d}\lambda_{H}(g) = \mathcal{C}_{H}^{K}(U,V).
\end{equation*}
Let $\widetilde{f}$ be the trivial extension of $f$. By Lemma \ref{restrictextendKbi}(ii), $\widetilde{f} \in \mathcal{F}_{G}^{K}(U,V)$. We claim that $\widetilde{f}$ is an extremal function for $\mathcal{C}_{G}^{K}(U,V)$. Indeed, 
\begin{equation*}
\int_{G}\widetilde{f}(g)\textup{d}\lambda_{G}(g) = \int_{H}f(g)\textup{d}\lambda_{H}(g) = \mathcal{C}_{H}^{K}(U,V) = \mathcal{C}_{G}^{K}(U,V).
\end{equation*}
\end{proof}
Theorem \ref{generalKtheorem}, together with Theorem \ref{turanconstants} and Theorem \ref{support}, implies the following theorem.
\begin{theorem} Let $G$ be an amenable locally compact group, $K$ a compact subgroup of $G$, $U$ a closed $K$-bi-invariant symmetric neighbourhood of the identity \(e \in G\) having finite Haar measure, and $V$ a closed $K$-bi-invariant symmetric subset of $G$. Then there exists an extremal function for the Delsarte constant $\mathcal{C}_{G/K}(U,V)$.   
\end{theorem}
Specialising to the sphere $\mathbb{S}^{d} \cong \mathrm{SO}(d+1)/\mathrm{SO}(d)$, and noting that \(\mathrm{SO}(d+1)\) is amenable on account of being compact, have the following corollary.
\begin{corollary} There exists an extremal function for the Tur\'an constant   $\mathcal{T}_{\mathbb{S}^d}(c)$ in \eqref{spherturan}.
\end{corollary}
We close this section by observing that since every locally compact Abelian group is amenable, Theorem \ref{generalKtheorem} generalises the existence results in \cite{berdrevram, ramabulana}.

\section{convolution roots}
Following \cite{ziegel}, for kernels $\Phi_1,\Phi_2: \mathbb{S}^d \times \mathbb{S}^d \to \mathbb{R}$, their \textit{spherical convolution} is defined to be
\begin{equation}\label{sphericalconvolution1}
    \Phi_1\ast \Phi_2(x,y):= \int_{\mathbb{S}^d}\Phi_1(x,z)\Phi_2(z,y)\textup{d}z
\end{equation}
for all $x,y \in \mathbb{S}^d$. A kernel $\Phi: \mathbb{S}^d \times \mathbb{S}^d \to \mathbb{R}$ is said to have a \textit{spherical convolution root} if there exists a kernel $\Psi: \mathbb{S}^d \times \mathbb{S}^d \to \mathbb{R}$ such that $\Phi = \Psi \ast \Psi$. In \cite{ziegel}, the question of the existence of an isotropic spherical convolution root for a positive definite kernel on the sphere is considered. In particular, the following was proved.
\begin{theorem}[See {\cite[Theorem 3.2]{ziegel}}] Any continuous isotropic positive definite kernel $\Phi: \mathbb{S}^d \times \mathbb{S}^d \to \mathbb{R}$ has an isotropic spherical convolution root $\Psi \in L^2(\mathbb{S}^d \times \mathbb{S}^d)$.
\end{theorem}

In this section, we consider a similar question in the more general setting of compact Gelfand pairs.

Recall that for kernels $\Phi_1, \Phi_2: G/K \times G/K \to \mathbb{R}$, their spherical convolution is
\begin{equation}\label{spheicalconvolution2}
    \Phi_1\ast \Phi_2(xK,yK):= \int_{G/K}\Phi_1(xK,zK)\Phi_2(zK,yK)\textup{d}\mu (zK)
\end{equation}
for all $x, y \in G$. Clearly, when $G = \mathrm{SO}(d)$ and $K = \mathrm{SO}(d-1)$, then $G/K \cong \mathbb{S}^d$ and \eqref{spheicalconvolution2} reduces to \eqref{sphericalconvolution1}. So, we consider the question of whether for any continuous positive definite $G$-invariant kernel $\Phi: G/K \times G/K \to \mathbb{R}$ there exists a $G$-invariant kernel $\Psi: G/K \times G/K \to \mathbb{R}$ such that $\Phi = \Psi \ast \Psi$. We shall answer this question in the case where $(G,K)$ is a compact Gelfand pair. Note that $(\mathrm{SO}(d),\mathrm{SO}(d-1))$ is a compact Gelfand pair, so our result contains that of \cite{ziegel}. The argument in \cite{ziegel} is specific to the unit sphere $\mathbb{S}^d$, while our argument is different and works for all compact Gelfand pairs.

\begin{theorem}
    Let $(G,K)$ be a compact Gelfand pair, then for any continuous positive definite $K$-bi-invariant function $f: G \to \mathbb{R}$, there exists a positive definite $K$-bi-invariant function $g: G \to \mathbb{C}$ in \(L^2(G)\) such that $f = g\ast g$.
\end{theorem}

\begin{proof}
    Since $G$ is compact and $f$ is continuous, we have that $f \in L^1(G)$. Then $f \in L^1(Z)$ by Theorem \cite[Theorem 6.4.5]{dijk}. Since $\widehat{f}$ is non-negative, the square root $h=\sqrt{\widehat{f}}$ is well-defined and is in $L^2(Z)$. Now, let $g: G \to \mathbb{C}$ be the inverse Fourier transform of $h$, which exists by \cite[Theorem 6.4.6]{dijk}. Since $G$ is compact and hence of finite Haar measure, the Cauchy-Schwartz inequality implies that $g$ is in $L^1(G)$. It follows that the Fourier transform of $g$ is defined everywhere on $Z$, and is equal to $h$.

    Now, since $g \in L^{1}(G)$, we have that $\widehat{g \ast g} = \widehat{g}\cdot \widehat{g}$ by \cite[Theorem 6.4.3]{dijk}. Hence $\widehat{g \ast g} = h^2 = \widehat{f}$. But $\widehat{f} \in L^2(Z)$ because $f \in L^2(G)$. Hence, taking the $L^2$ inverse Fourier transform on both sides of $\widehat{g \ast g} = \widehat{f}$, we have that $f = g \ast g$ in $L^2(G)$. But both $f$ and $g \ast g$ are continuous, so $f = g \ast g$ everywhere on $G$.
\end{proof}
By Theorem \ref{correspondence} and Theorem \ref{preservesconvolution}, we have the following corollary.
\begin{corollary}
    Let $(G,K)$ be a compact Gelfand pair, then for any continuous positive definite kernel $\Phi: G/K \times G/K \to \mathbb{R}$ there exists a $G$-invariant positive definite kernel $\Psi: G/K \times G/K \to \mathbb{R}$ such that $\Phi = \Psi \ast \Psi$.
\end{corollary}

\section{Acknowledgements}
The author thanks the University of Cape Town's Shuttleworth Postgraduate Scholarship, Faculty of Science Transformation Fund, and the Science Faculty Research Commitee (through Postgraduate Publishing Incentive), for the financial support provided towards the author's doctoral studies, during which this manuscript was prepared. The author is also thankful to Professor Tilmann Gneiting for providing information regarding the Tur\'an problem. Finally, the author is grateful to Associate Professor Elena E. Berdysheva and Professor Szil\'ard Gy. R\'ev\'esz for the helpful mathematical discussions and careful reading of this manuscript.
\printbibliography
\end{document}